\documentclass[11pt, a4paper]{amsart}
\usepackage[a4paper]{geometry}
\usepackage[utf8]{inputenc}
\usepackage{amsmath}
\usepackage{amssymb}
\usepackage{amsthm}
\usepackage{enumerate}
\usepackage{graphicx}
\usepackage{import}
\usepackage{xifthen}
\usepackage{pdfpages}
\usepackage{tikz} 
\usepackage{tikz-cd} 
\usepackage{mathabx}
\usepackage{hyperref}
\usepackage{thmtools, thm-restate}

\usepackage[foot]{amsaddr}

\usetikzlibrary{shapes.geometric}
\usetikzlibrary{graphs,graphs.standard}
\usetikzlibrary{fit}
\usetikzlibrary{calc}

\usepackage[giveninits]{biblatex}
\theoremstyle{plain}
\newtheorem{theorem}{Theorem}[section]
\newtheorem{claim}{Claim}[theorem]

\newtheorem{lemma}[theorem]{Lemma}
\newtheorem{corollary}[theorem]{Corollary}

\theoremstyle{definition}
\newtheorem{definition}[theorem]{Definition}
\newtheorem{example}[theorem]{Example}
\newtheorem{proposition}[theorem]{Proposition}
\newtheorem{question}[theorem]{Question}
\newtheorem{conjecture}[theorem]{Conjecture}

\theoremstyle{remark}

\newcommand{\Z}{\mathbb{Z}}

\newcommand{\Ps}{\operatorname{PStab}}

\title{Quasi-convex Splittings of Acylindrical Graphs of Locally finite height Groups}
\author{William D Cohen}
\address{Centre for Mathematical Sciences, University of Cambridge, 
Cambridge, CB3 0WA}
\email{wdc26@cam.ac.uk}

\AtBeginBibliography{\small}

\begin{document}

\begin{abstract}
    We find a condition on the acylindrical action of a finitely presented group on a simplicial tree which guarantees that this action will be dominated by an acylindrical action with finitely generated edge stabilisers, and find the first example of an action of a finitely presented group where there is no such dominating action. As a consequence, we show that any finitely presented group that admits a decomposition as an acylindrical graph of (possibly infinitely generated) free groups is virtually compact special, and that every finitely generated subgroup of a one-relator group with an acylindrical Magnus hierarchy is virtually compact special.
\end{abstract}

\maketitle

\section{Introduction}
Let $G$ be a one-relator group. One of the oldest and most effective tools for studying such groups is a \emph{Magnus Hierarchy}, an iterative HNN-structure on $G$ whose vertex groups are one-relator groups and whose edges are free subgroups generated by subsets of the generators of the vertices (c.f. \ref{def:Magnus}). In his paper \cite{Linton25}, Linton proved that if a Magnus hierarchy for $G$ is \emph{acylindrical} (c.f. \ref{Def:introkc}) then $G$ is hyperbolic and the hierarchy is \emph{quasi-convex} (c.f. \ref{def:qc}), a much stronger condition. 

Now assume that $G$ has such an acylindrical hierarchy, and let $H$ be a finitely generated subgroup of $G$. By coherence of $G$ \cite[Theorem~1.1]{JaikinLinton25}, Gersten's theorem \cite[Theorem~5.4]{Gersten96}, and the fact that one-relator groups have virtual cohomological dimension at most two \cite[Corollary~9]{Howie84}, the group $H$ is hyperbolic, so a natural question to ask is whether $H$ will itself admit a quasi-convex hierarchy. Indeed, $H$ will always admit an acylindrical hierarchy induced by the hierarchy for $G$, but this need not be quasi-convex, the most immediate issue being that we cannot assume that it has finitely generated edge stabilisers --- see Section~\ref{sec:Applications} for an example of such a $G$ and $H$ where the induced hierarchy on $H$ has infinitely generated edge stabilisers. In this paper we prove that even in such a case, we may construct a quasi-convex hierarchy for $H$.

The definition of an acylindrical action on a tree was first formulated by Sela \cite{Sela97}, and later generalised by Weidmann \cite{Weidmann12} to the following. 

\begin{definition}\label{Def:introkc} Let $G$ be a group acting by simplicial isometry on some simplicial tree $T$ and let $k\geq0$ and $C>0$ be integers. We say that the action of $G$ on $T$ is $(k, C)$\emph{-acylindrical} if the pointwise stabiliser of any edge path in $T$ of length at least $k$ contains at most $C$ elements.
\end{definition}
We say that a group $G$ is \emph{acylindrically arboreal} if $G$ admits an acylindrical action on a tree that is \emph{non-elementary} (c.f. Theorem~\ref{Thm:categorisedActions}). Acylindrical arboreality is a special case of acylindrical hyperbolicity, a well studied and powerful generalisation of hyperbolicity that has been of much interest in recent years \cite{MinasyanOsin15, Osin16,  Osin17}, and was the first such generalisation to encompass mapping class groups of hyperbolic surfaces (\cite{Bowditch08,MasurMinsky99}, see also \cite[Section~8]{Osin16}).

A concept strongly related to an acylindrical action on a tree is that of a \emph{quasi-convex splitting}, an action of a hyperbolic group $G$ on a simplicial tree where the edge stabilisers are all quasi-convex in $G$ (c.f Definition~\ref{def:qc}). A quasi-convex splitting with finitely many orbits of edges will always be acylindrical by a result of Gitik, Mj, Rips and Sageev \cite[Main Theorem]{GitikMahanRipsSageev98} that quasi-convex subgroups have \emph{finite height} (c.f. Definition~\ref{def:finheight}), meaning that in an action on a tree with quasi-convex edge groups any large enough set of edges will have a finite stabiliser. The converse is more difficult, and is the subject of the following conjecture that can be found in question form in \cite{BestvinaList}.

\begin{conjecture} \label{CONJECTURE}
    Let $G$ be a hyperbolic group, and $H$ a finitely generated subgroup of $G$ with finite height. Then $H$ is quasi-convex in $G$.
\end{conjecture}

Our opening question is a special case of the following natural question.

\begin{question} \cite[Question~1.7]{Cohen25}\label{Q1}
    Let $G$ be a hyperbolic group with a non-elementary acylindrical action on a tree. Does $G$ admit a quasi-convex splitting?
\end{question}

 One of the main reasons for the significance of quasi-convex splittings is their link to virtual specialness of hyperbolic groups. In particular, in \cite[Theorem~13.1]{WiseQCBook}, Wise proves that if a hyperbolic group admits a quasi-convex splitting with virtually compact special vertex stabilisers then it itself is virtually compact special. Virtual specialness is an important concept in modern geometric group theory, and its definition along with the fact that it is inherited from vertex groups in a quasi-convex splittings was instrumental in proving the Virtual Haken Conjecture, one of the most influential group theoretic results in recent years, see \cite{Agol13}. 
 
 In general, finding quasi-convex splittings of a hyperbolic group can be very challenging. On the other hand, acylindricity is a \emph{a priori} much weaker condition as we will see in Section~\ref{sec:fhqc}, and so a positive answer to Question~\ref{Q1} would be useful in this respect. 

Let $G$ be a group, and let $T$ and $T'$ be simplicial trees on which $G$ acts by simplicial isometry. We say that the action of $G$ on $T'$ \emph{dominates} the action on $T$ if there exists a $G$-equivalent graph epimorphism $\phi:T'\rightarrow T$. Using methods that rely on constructing dominating actions, Question~\ref{Q1} has been answered in the affirmative for many important classes of hyperbolic groups, including:

\begin{enumerate}
    \item The fundamental groups of closed and orientable hyperbolic $3$-manifolds \cite[Theorem~1.8]{Cohen25}; and
    \item Free groups, and more generally hyperbolic locally quasi-convex groups such as hyperbolic limit groups \cite[Corollary~3.12]{Wilton08} and many small cancellation groups \cite{McCammondWise05, McCammondWise08}.
\end{enumerate}

As stated above, the first obstacle to overcome when attempting to construct quasi-convex splittings from acylindrical ones is that a generic acylindrical splitting need not have finitely generated stabilisers. In \cite{Cohen25}, this problem was solved for hyperbolic $3$-manifold groups using a construction of Stallings, Epstein and Waldhausen which takes a splitting of such a group and outputs a collection of embedded surfaces one may cut along to obtain a new dominating action on a tree. This construction was proved to preserve acylindricity (and indeed promote it to quasi-convexity) in \cite{Cohen25} using the Subgroup Tameness Theorem for hyperbolic $3$-manifolds (see \cite[Theorem~5.2]{AschenbrennerFriedlWilton12}, for example).

A similar but weaker method exists in general for finitely presented groups. In particular, there is the construction due to Dunwoody and Sageev known as the \emph{Dunwoody--Sageev resolution} \cite[Lemma~2.2]{DunwoodySageev99} that takes as input an action of a finitely presented group on a tree and outputs a dominating action with finitely generated edge stabilisers that are in some sense geometric, in that they arise as one-dimensional subspaces of a Cayley complex. For a locally quasi-convex group, the actions produced by the Dunwoody--Sageev resolution will always have quasi-convex edge stabilisers, answering Question~\ref{Q1} in this case.

The main technical result of this paper exhibits a large class of actions in which acylindricity is preserved by the Dunwoody--Sageev Resolution, and indeed where the resolution produces finite height splittings that dominate the original acylindrical splitting. We say that a group $G$ is \emph{locally finite height} if every finitely generated subgroup of $G$ has finite height in $G$, and that an action of $G$ on a simplicial tree $T$ is \emph{finite height} if every edge stabiliser of this action has finite height in $G$. The following is our main theorem.

\begin{theorem} \label{thm:mainintro}
Let $G$ be a finitely presented group with bounded finite subgroups that acts minimally, acylindrically and non-elementarily on a tree $T$. 
    
    \begin{enumerate}[a.]
        \item If all vertex stabilisers of the action of $G$ on $T$ are locally finite height then $G$ admits a non-elementary finite height acylindrical splitting with finitely generated edge stabilisers that dominates the action of $G$ on $T$. Thus, assuming Conjecture~\ref{CONJECTURE}, if $G$ is hyperbolic then $G$ admits a quasi-convex splitting that dominates the action of $G$ on $T$.
        \item If there exists an edge stabiliser that is finite height acylindrical in $G$, and this edge stabiliser is locally finite height, then $G$ admits a non-elementary finite height acylindrical splitting with finitely generated edge stabilisers. Furthermore, if every edge stabiliser satisfies this condition then we may assume that the new action dominates the action of $G$ on $T$. Thus, assuming Conjecture~\ref{CONJECTURE}, if $G$ is hyperbolic then $G$ admits a quasi-convex splitting that dominates the action of $G$ on $T$.
    \end{enumerate}
\end{theorem}
We present several applications of this theorem. First, all countable free groups are locally finite height, and thus every finitely presented group that is an acylindrical graph of free groups will admit a finite height splitting with finitely generated edge stabilisers. Indeed, finitely generated free groups are known to be locally quasi-convex by Hall's Theorem \cite{Hall_1949}, and so when dealing with hyperbolic graphs of free groups one may often obtain quasi-convexity without having to appeal to Conjecture~\ref{CONJECTURE}. The following corollary follows immediately from Theorem~\ref{thm:mainintro}(a) and an application of the Quasi-Convex Combination Theorem of Kapovich \cite[Theorem~1.2]{Kapovich01}, which states that if a group $G$ acts acylindrically on a tree with each vertex stabiliser hyperbolic and each edge stabiliser quasi-convex in its neighbouring vertex stabilisers, then $G$ is hyperbolic each edge stabiliser will be quasi-convex in $G$.

\begin{corollary}\label{cor:immediateIntro}
    Let $G$ be a finitely presented group, and assume that $G$ acts acylindrically on some tree $T$ such that the vertex stabilisers of this action are free. Then $G$ is hyperbolic and admits a quasi-convex splitting with free vertex groups, and so $G$ is virtually compact special by \cite[Theorem~13.1]{WiseQCBook}.
\end{corollary}

Our methods also apply to subgroups of hyperbolic groups with quasi-convex splittings over locally finite height groups. The following corollary allows us to obtain virtual compact specialness for many groups without appealing to Conjecture~\ref{CONJECTURE}.

\begin{corollary}\label{cor:introStrong}
    Let $G$ be a hyperbolic group, and let $H$ be a hyperbolic subgroup of $G$.
    
    \begin{enumerate}[(a)]
        \item If $G$ admits an action on a tree $T$ that is a quasi-convex splitting over a locally quasi-convex group, and $H$ the induced action of $H$ on $T$ is non-trivial, then $H$ admits a non-trivial quasi-convex splitting that dominates the minimal subtree of the induced action of $H$ on $T$.
        \item If $G$ admits a quasi-convex hierarchy with locally quasi-convex edge stabilisers then $H$ admits a quasi-convex hierarchy, and therefore $H$ is virtually compact special by \cite[Theorem~13.1]{WiseQCBook}.
    \end{enumerate}
\end{corollary}

In the case where $G$ has virtual cohomological dimension $2$, and so we can apply \cite[Theorem~5.4]{Gersten96}, we need only ask for $H$ to be finitely presented. For example, the following settles our opening question.
\begin{corollary}\label{cor:intro}
      Let $H$ be a finitely generated subgroup of a one-relator group $G$ such that $G$ admits an acylindrical Magnus hierarchy $G_0\hookrightarrow G_1\hookrightarrow\cdots\hookrightarrow G_n\cong G$. Then $H$ admits a a quasi-convex hierarchy, and so is virtually compact special by \cite[Theorem~13.1]{WiseQCBook}. In particular the group $H$ is acylindrically arboreal if and only if it admits a non-trivial a quasi-convex splitting.
     
     Furthermore, $H$ admits a quasi-convex hierarchy and in particular is virtually compact special by  \cite[Theorem~13.1]{WiseQCBook}.
\end{corollary}

This corollary will apply to subgroups of many two-generator one-relator groups, such as the one-relator group given in Example~\ref{ex:BNS}, and subgroups of all one-relator groups with negative immersions \cite[Corollary~6.9]{Linton25}.

In examples (1) and (2) given above and in Corollary~\ref{cor:immediateIntro} the constructed quasi-convex splitting can be assumed to dominate the original acylindrical splitting, and the splittings produced by Theorem~\ref{thm:mainintro} will also often dominate the original given splitting. Therefore, a particularly interesting strengthening of Question~\ref{Q1} is the following.

\begin{question} \label{Q2}
    Is every non-elementary acylindrical action of a hyperbolic group $G$ on a tree dominated by a quasi-convex splitting?
\end{question}

A positive answer to Question~\ref{Q2} would have strong consequences. In particular, again using \cite[Theorem~13.1]{WiseQCBook}, a positive answer would imply that every hyperbolic group that admits an acylindrical hierarchy in fact admits a quasi-convex hierarchy, and therefore is virtually compact special with a proof analogous to that of Corollary~\ref{cor:intro}. This would provide a significantly weaker condition for virtual specialness of a hyperbolic group than we currently have at our disposal.

We finish by illustrating the issues that arise when attempting to relax the hypotheses of Theorem~\ref{thm:mainintro}. In Section~\ref{sec:Broken} we will show that the Dunwoody--Sageev resolution cannot be assumed to preserve the acylindricity of a general action with the following proposition.

 \begin{proposition}\label{prop:introBadExample}
    There exists a (non-hyperbolic) finitely presented group $G$ such that:
    \begin{enumerate}
        \item $G$ acts acylindrically and non-elementarily on a tree $T$ with infinitely generated edge stabilisers; and
        \item if $T'$ is a tree on which $G$ acts acylindrically with $T'$ dominating $T$ then the action of $G$ on $T'$ has at least one infinitely generated edge stabiliser.
    \end{enumerate}
    In particular, the Dunwoody--Sageev resolution will not preserve acylindricity in general.
 \end{proposition}

Somewhat unfortunately, the group constructed does admit an acylindrical action with finitely generated edge stabilisers. To counter this, we ask the following.

 \begin{question}\label{q:mine}
    Does every acylindrically arboreal finitely presented group  admit a non-elementary acylindrical action on a simplicial tree with finitely generated edge stabilisers?
 \end{question}

\subsection{Acknowledgements}
This work was completed while the author was a PhD student at the University of Cambridge, supervised by Jack Button, and the author is very grateful to his supervisor for all of his help. The author would also like to thank Henry Wilton for several very helpful discussions, and Daniel Groves for his helpful suggestions. Finally, the author is very grateful to an anonymous reviewer, whose careful reading and helpful suggestions allowed for a much shortened proof of our main theorem, and for Corollaries~\ref{cor:introStrong} and~\ref{cor:intro} to be proved independently from Conjecture~\ref{CONJECTURE}.

Financially, support from the Cambridge Trust Basil Howard Research Graduate Studentship is gratefully acknowledged. 

\section{Preliminaries - Acylindrical Hyperbolicity and Acylindrical Arboreality}
We will assume the reader has some familiarity with Bass--Serre theory, and for a more detailed discussion we refer to \cite{stilwell2002trees, dicks2011groups}. We define a \emph{graph of groups} $(\Gamma, \mathfrak{G})$ to be a connected directed graph $\Gamma$ that may not be finite or simple along with $\mathfrak{G}$, which is the following data:
\begin{itemize}
    \item To every vertex $v\in V(\Gamma)$ we assign a \emph{vertex group} $G_v$, and to every edge $e\in E(\Gamma)$ we assign an \emph{edge group} $G_e$;
    \item To every edge $e\in E(\Gamma)$ we assign monomorphisms $d_0:G_e\rightarrow G_{i(e)}$ and $d_1:G_e\rightarrow G_{t(e)}$, where $i(e)$ and $t(e)$ are the initial and terminal vertices of $e$ in $\Gamma$ respectively.
\end{itemize}
We define the fundamental group $\pi_1(\Gamma, \mathfrak{G})$ of $(\Gamma, \mathfrak{G})$ as in \cite[Section~I.5.1]{stilwell2002trees}.

We will use a slight abuse of notation to consider each vertex group $G_v$ as a subgroup of the fundamental group $\pi_1(\Gamma, \mathfrak{G})$ along the natural inclusion. Similarly, we will consider each edge group $G_e$ to be the subgroup of the fundamental group given by the image of $d_0(G_e)$ in the vertex group $G_{i(e)}$. We call a graph of groups \emph{trivial} if there exists some $v\in V(\Gamma)$ such that $G_v= \pi_1(\Gamma, \mathfrak{G})$, or \emph{non-trivial} otherwise. We say that a graph of groups $(\Gamma, \mathfrak{G})$ is a \emph{graph of groups decomposition} or \emph{splitting} of a group $G$ if the fundamental group $\pi_1(\Gamma, \mathfrak{G})$ is isomorphic to $G$, and identify $G$ with $\pi_1(\Gamma, \mathfrak{G})$. We denote by $T(\Gamma, \mathfrak{G})$ the \emph{Bass--Serre tree} associated to the splitting, on which $G$ acts naturally by isometry with respect to the edge metric and without inversion as in \cite[Section~I.5.3]{stilwell2002trees}. 

From now on we will assume that all trees are simplicial, and all actions on trees are by simplicial isometry and without inversion, so that every action on a tree is associated to a \emph{quotient graph of groups} as in \cite[Section~I.5]{stilwell2002trees}, whose edge and vertex groups correspond to conjugacy classes of stabilisers of the original group action. We say that an action of a group $G$ on a tree $T$ is \emph{minimal} if there exists no proper subtree of $T$ that is fixed setwise by the action of $G$. We have the following lemma, which will allow us to assume that the actions we consider are minimal in most cases.

\begin{lemma}\label{lem:minimal}\cite[Lemma~I.4.13]{dicks2011groups}
    Let $G$ be a finitely generated group acting on a tree $T$. Then there exists a unique minimal subtree $T_m\subseteq T$, and the induced action of $G$ on $T_m$ has finitely many orbits of edges.
\end{lemma}

The following lemma of Bieri will be essential in dealing with finite generation of stabilisers of actions on trees.

\begin{lemma} \label{lem:bieri}\cite[Proposition~2.13(a) and (b)]{bieri1981homological}
    Let $G$ be a finitely generated group acting on a tree $T$ with finitely many orbits of edges. If every edge stabiliser of this action is finitely generated, then so is every vertex stabiliser.
\end{lemma}

A common generalisation of a graph of groups is that of a hierarchy of groups, or an iterated graph of groups decomposition.

\begin{definition}\label{def:hier}
    Let $G$ be a group and $n$ a non-negative integer. For all $0\leq i< n$ let $J_i$ be an indexing set that may not be finite, and let $J_n=\{1\}$. A \emph{hierarchy decomposition for $G$ of length $n$} is a collection of groups $\{G_{i, j}\}_{0\leq i\leq n, j\in J_i}$ such that
    \begin{enumerate}[(H1)]
        \item $G\cong G_{n, 1}$.
        \item For all $j\in J_0$, $G_{0, j}$ is finite.
        \item For all $0< i\leq n$, for all $j\in J_i$, the group $G_{i, j}$ admits an action on a tree $T_{i, j}$ whose vertex stabilisers are groups of the form $G_{i-1, k}$ for some $k\in J_{i-1}$.
    \end{enumerate}
\end{definition}

There are various properties of different strengths that we can ask for a hierarchy or action on a tree to satisfy, where if (P) is some property of a graph of groups decomposition, we say that a hierarchy decomposition $\{G_{i, j}\}$ for a group $G$ has property $P$ if the action of $G_{i, j}$ on $T_{i, j}$ has property (P) for all $i, j$. In this paper, we are particularly interested in the following properties. 

Intuitively, an action is \emph{acylindrical} if any long enough path in the tree has a \emph{uniformly} bounded pointwise stabiliser. Formally, we define acylindricity as follows.

\begin{definition}\label{Def:kc} Let $G$ be a group acting by simplicial isometry on some simplicial tree $T$ and let $k\geq0$ and $C>0$ be integers. We say that the action of $G$ on $T$ is $(k, C)$\emph{-acylindrical} if the pointwise stabiliser of any edge path in $T$ of length at least $k$ contains at most $C$ elements. We say the action of $G$ on $T$ is \emph{acylindrical} if there exist $k$ and $C$ such that the action is $(k, C)$-acylindrical, and that such an action is \emph{non-elementary} if the action of $G$ on $T$ is non-trivial and does not fix any bi-infinite geodesic setwise.

If a group $G$ admits a non-elementary acylindrical action on some tree $T$ then we say that $G$ is \emph{acylindrically arboreal}.
\end{definition}

\begin{example}
    Several interesting examples of acylindrically arboreal groups may be found in the literature, including:
    \begin{enumerate}
    \item The free product $G$ of any two non-trivial groups induces a $(1, 1)$-acylindrical action on its corresponding Bass--Serre tree, which will be non-elementary unless $G\cong D_\infty$.
    \item For a more complicated example, a right angled Artin group is acylindrically arboreal if and only if the underlying graph has graph theoretical diameter at least three \cite[Corollary~3.12]{Cohen25}.
    \item If $G$ is the fundamental group of a closed non-geometric $3$-manifold, then Wilton and Zalesskii prove that $G$ acts acylindrically on the tree corresponding to its JSJ-decomposition \cite[Lemma~2.4]{WiltonZalesskii10}. In the hyperbolic case, the fundamental group $G$ of a finite volume orientable hyperbolic $3$-manifold will be acylindrically arboreal if and only if the manifold contains an embedded $2$-sided closed incompressible surface whose fundamental group includes into $G$ as a geometrically finite subgroup \cite[Theorem~1.6]{Cohen25}.
    \end{enumerate}
\end{example}
By a result first explicitly proved in \cite[Theorem~2.17]{Cohen25}, but
essentially due to Minasyan and Osin \cite[Lemma~4.2]{MinasyanOsin15}, an action on a simplicial tree is non-elementary acylindrical as in Definition~\ref{Def:kc} if and only if it is non-elementary acylindrical in the sense of Bowditch \cite{Bowditch08}. Thus, if a group is acylindrically arboreal then it is acylindrically hyperbolic, and we have the following powerful classification theorem.
\begin{theorem}
    \textit{\cite[Theorem~1.1]{Osin16}}\label{Thm:categorisedActions} Let $G$ be a group acting acylindrically on a tree $T$. Then $G$ satisfies exactly one of the following conditions.
    \begin{enumerate}
        \item The action of $G$ on $T$ has a global fixed point. In this case, we say that the action of $G$ on $T$ is \emph{elliptic}.
        \item The group $G$ is virtually cyclic and the action of $G$ on $T$ fixes some bi-infinite geodesic in $T$ setwise. In this case we say that the action of $G$ on $X$ is \emph{lineal}.
        \item Neither of the previous cases occur. In this case we call the action \emph{non-elementary}.
    \end{enumerate}       
\end{theorem}
We will also require the following lemma.
\begin{lemma} \cite[Corollary~1.5]{Osin16}\label{lem:AAnormal}
    Let $G$ be an acylindrically arboreal group. Then every infinite normal subgroup of $G$ is acylindrically arboreal. In particular, such a subgroup cannot be abelian.
\end{lemma}

\subsection{finite height and Quasi-Convex Splittings}\label{sec:fhqc}
As in Question~\ref{Q1}, we are concerned with whether we can use acylindrical splittings to build splittings with stronger properties. In particular, we are concerned with \emph{finite height} and \emph{quasi-convex} splittings of hyperbolic groups. 

Intuitively, a finite height splitting is an action on a tree where a set of edges containing enough elements has a finite stabiliser. It is important to note that, unlike acylindricity, this definition does not require that these stabilisers are uniformly bounded, only that they are finite. More formally, we define finite height as follows.

\begin{definition} \label{def:finheight}
    Let $G$ be a group, and $H\leq G$ a subgroup. For $g_1, g_2\in G$, we say that the conjugates $g_1Hg_1^{-1}$ and $g_2Hg_2^{-1}$ are \emph{essentially distinct} if $g_1H\neq g_2H$.

    For some natural number $n$, we say that $H$ has $\emph{height at most n}$ in $G$ if the intersection of any collection of at least $n$ essentially distinct conjugates of $H$ in $G$ is finite. We define the \emph{height} of $H$, $h_G(H)$ to be the smallest $m$ such that $H$ has height at most $m$ in $G$, or to be infinite if no such $m$ exists.
\end{definition}
\begin{example}
    Some of the simplest examples of finite height subgroups are finite index subgroups. If $G$ is a group and $H\leq G$ has finite index then $H$ will have height at most $[G:H]+1$, owing to the fact that no collection of $[G:H]+1$ left cosets of $H$ in $G$ can be disjoint.
\end{example}

We say that the action of a group $G$ on a tree $T$ is a \emph{finite height splitting} if for all edges $e\in E(T)$, the $G$-stabiliser $G_e$ of $e$ is finite height in $G$, i.e $h_G(G_e)$ is finite. As above, some care should be taken when comparing acylindricity to finite height splittings owing to the different restrictions on the finite stabilisers. In particular, while in an acylindrical action stabilisers of long enough paths must be uniformly bounded, in a finite height splitting we only require that stabilisers of long enough sets are finite. The following example demonstrates this difference.

\begin{example}
    Let $G=\bigcup_{i\in\mathbb{N}}\Z/2^i\Z$, a directed union of cyclic groups of order $2^i$. Let $T$ be the tree whose vertex set is labelled by the disjoint union of cosets \[V(T)=\bigsqcup_{j\in\mathbb{N}}\left( \frac{G}{\Z/2^i\Z}\right),\]
    and where each coset $g(\Z/2^i\Z)$ is connected by an edge to every coset $g(\Z/2^{i+1}\Z)$.

    The group $G$ acts on $T$ via the left coset action on the labels of its vertices, and the edge stabilisers of this action are all isomorphic to finite groups, so naturally have height $1$. This action is therefore finite height, but will not be acylindrical, as for all $k$ we can find a path of length $k$ whose stabiliser is contains $2^i$ elements for any $i$ by taking the path that starts at the vertex labelled $(\Z/2^i\Z)$ and finishing at the vertex labelled $(\Z/2^{i+k}\Z)$.
\end{example}

In almost all settings this will not cause issue, but to work in generality we will sometimes need to assume that we do not have arbitrarily large finite subgroups. As such, we say that a group $G$ has \emph{bounded finite subgroups} if there exists some natural number $M$ such that for all finite subgroups $H_f$ of $G$, $|H_f|<M$.

The following lemma follows by definition.
\begin{lemma}\label{lem:FHtoAA}
    Let $G$ be a group with bounded finite subgroup, and assume that $G$ acts on a tree $T$ with finitely many orbits of edges. If the action of $G$ on $T$ is a finite height splitting then it is also an acylindrical splitting.
\end{lemma}
As stated above, the conditions in the statement of Lemma~\ref{lem:FHtoAA} are not overly restrictive. Indeed, these conditions can always be assumed to hold for hyperbolic groups --- a hyperbolic group is finitely generated so any action on a tree admits a unique minimal subtree with finitely many orbits of edges by Lemma~\ref{lem:minimal}, and a hyperbolic group will always have bounded finite subgroups by \cite[Theorem]{Bogopolskii95}, for example. Motivated by Lemma~\ref{lem:FHtoAA} we make the following definition.

\begin{definition}
    Let $G$ be a group with bounded finite subgroups, and assume that $G$ admits a finite height splitting with finitely many orbits of edges. This splitting will be acylindrical by Lemma~\ref{lem:FHtoAA}, so we say that this splitting is \emph{non-elementary} if it is non-elementary in the sense of Theorem~\ref{Thm:categorisedActions}
\end{definition}

The concept of a finite height splitting is strictly stronger than that of an acylindrical one in the setting of groups with bounded finite subgroups acting on trees with finitely many orbits of edges. Consider the following example.
\begin{example}
    Let $G$ be the right angled Artin group on the $4$-path $P_4$, so \[G\cong\langle a, b, c, d\mid [a, b], [b, c], [c, d]\rangle. \] Then by \cite[Corollary~3.12]{Cohen25}, the splitting given by $\langle a, b, c\rangle*_{\langle b, c\rangle}\langle b, c, d\rangle$ is non-elementary acylindrical, but the edge group does not have finite height in $G$. Indeed, the conjugate of $\langle b, c\rangle$ by any power of $a$ will contain the span of $b$, so the height of $\langle b, c\rangle$ in $G$ will be infinite.

\end{example}
In this example, finite height fails ``locally'' --- the edge group is not finite height in the vertex groups into which it includes. It turns out that this is the only way in which an acylindrical splitting can fail to be finite height, as shown by the following lemma.
\begin{lemma}\label{LFHtoFH}
   Let $G$ be a group acting $(k, C)$-acylindrically on some tree $T$ with finitely many orbits of edges. Then this action is a finite height splitting if and only if each edge group has finite height in both of its neighbours, i.e. for all $e\in E(\Gamma)$, $h_{G_{i(e)}}(d_{e, 0}(G_e))<\infty$ and $h_{G_{t(e)}}(d_{e, 1}(G_e))<\infty$.
\end{lemma}
\begin{proof}
    Firstly, for the if direction assume that the action of $G$ on $T$ is a finite height splitting. Then the stabiliser $G_e$ of each edge $e\in E(T)$ must have finite height in $G$, so in turn must have finite height in the vertex stabilisers of the endpoints of $e$ by definition.

    Now assume for the only if direction that for all $e\in E(T)$, $h_{G_{i(e)}}(d_{e, 0}(G_e))<\infty$ and $h_{G_{t(e)}}(d_{e, 1}(G_e))<\infty$, and let $\Gamma=T/G$. For every $v\in \Gamma$, let 
    \[n_v=\sum_{e\in E(\Gamma) \text{ : }i(e)=v}h_{G_{v}}(d_{e, 0}G_e))+\sum_{e\in E(\Gamma) \text{ : }t(e)=v}h_{G_{v}}(d_{e, 1}G_e)).\]
    Thus, if $\tilde{v}\in V(T)$ is a lift of some vertex $v\in V(\Gamma)$, then the stabiliser of any collection of at least $n_v$ edges adjacent to $\tilde{v}$ is finite by the pigeonhole principle.
    
    Define $n=\max_{v\in V(\Gamma)}n_v$. We claim that every edge stabiliser of the action of $G$ on $T$ has height at most $n^k$. Indeed, let $G_e$ be the stabiliser of the edge $e\in E(T)$ with respect to the action of $G$ on $T$, and let $g_1G_e,...,g_{n^k}G_e$ be a collection of disjoint cosets of $G_e$ in $G$. Let $S\subseteq T$ be the subtree given by the convex hull of the edges $g_1e,...,g_{n^k}G_e$, so \[\Ps_G(S)=\bigcap_{i=1}^{n^k}g_iG_eg_i^{-1}.\] 
    
    We separate into two cases. First, if the graph theoretical diameter of $S$ is at least $k$ then $|\Ps_G(S)|\leq C$ by the acylindricity assumption, so in particular is finite. For the second case, assume the graph theoretical diameter of $S$ is less than $k$. As the diameter of $S$ is less than $k$, $S$ must have some vertex $v$ of degree at least $n$ as a tree of diameter less than $k$ and maximum vertex degree $n-1$ has less than $(n-1)^k$ vertices, and so less than $(n-1)^k-1< n^k$ edges. Thus by construction of $n$ we have that $\Ps_G(S)=\bigcap_{i=1}^{n^k}g_iG_eg_i^{-1}$ is finite.

    Consequently in both cases $\bigcap_{i=1}^{n^k}g_iG_eg_i^{-1}$ is finite, and so this is always finite and $G_e$ has height at most $n^k$ as claimed. It follows that each edge stabiliser of the action of $G$ on $T$ has has height at most $n^k$, and therefore has finite height as required.
\end{proof}

If $G$ is hyperbolic, we will also be interested in quasi-convex splittings, which are an \textit{a priori} stronger concept.
\begin{definition}\label{def:qc}
    Let $G$ be a hyperbolic group, and $H\leq G$ a subgroup. We say that $H$ is \emph{quasi-convex} in $G$ if the map $H\hookrightarrow G$ is a quasi-isometric embedding for any finite generating sets of $G$ and $H$.
\end{definition}
\begin{example}
    Many examples of quasi-convex subgroups of hyperbolic groups exist in the literature. For example, since every bounded set is quasi-isometric to the point, all finite subgroups of any hyperbolic group are quasi-convex, and it is a well known result (see Bowditch \cite[Section~1]{Bowditch98B}, for example) that every two-ended subgroup of a hyperbolic group is quasi-convex. Further important examples of quasi-convex subgroup arise as geometrically finite subgroups of fundamental groups of closed and orientable hyperbolic $3$-manifolds, see \cite[Corollary~1.3]{Hruska10}.
\end{example}
We say that the action of a hyperbolic group $G$ on a tree $T$ is a \emph{quasi-convex splitting} if the $G$-stabiliser of every edge $e\in E(T)$ is quasi-convex in $G$. In such a splitting the vertex stabilisers are also strongly controlled. We have the following result of Bowditch.

\begin{lemma} \cite[Proposition~1.2]{Bowditch98B} \label{lem:BowditchVertex}
    Let $G$ be a hyperbolic group acting on a tree $T$ with quasi-convex edge stabilisers and finitely many orbits of edges. Then each vertex stabiliser of this action is also quasi-convex in $G$.
\end{lemma}

\begin{example}
    Many important examples of quasi-convex splittings are well known. Perhaps the most important example is the JSJ-decomposition of a hyperbolic group, which is the ``maximal'' splittings of a $1$-ended hyperbolic group along two ended subgroups \cite{Bowditch98B}. 
\end{example}

Generally, quasi-convex subgroups are very well behaved and have the following properties.

\begin{lemma}\label{lem:QClemma}
    Let $G$ be a hyperbolic group and $H_1, H_2\leq G$ quasi-convex subgroups of $G$. Then the following hold.
    \begin{enumerate}
        \item \cite[Proposition~3.7]{BridsonHaefliger} The subgroups $H_1$ and $H_2$ are hyperbolic, so in particular they are finitely generated.
        \item \cite[Proposition~3.9]{BridsonHaefliger} The intersection $H_1\cap H_2$ is a quasi-convex subgroup of $G$.
        
    \end{enumerate}
\end{lemma}

The link to finite height splittings, and therefore acylindrical splittings, comes from the following influential result of Gitik, Mj, Rips and Sageev. 

\begin{theorem} \label{thm:QCFH}\cite[Main Theorem]{GitikMahanRipsSageev98}
    Let $G$ be a hyperbolic group, and $H\leq G$ a quasi-convex subgroup of $G$. Then $H$ has finite height in $G$.
\end{theorem}
The following corollary is well known, although we include a brief proof for completeness.
\begin{corollary}
    Let $G$ be a hyperbolic group that is not virtually cyclic, and let the action of $G$ on some tree $T$ be a minimal and non-trivial quasi-convex splitting. Then this action is a non-elementary acylindrical splitting.
\end{corollary}
\begin{proof}
    By Theorem~\ref{thm:QCFH} the action of $G$ on $T$ is finite height, and so recalling that minimal actions of hyperbolic groups on trees will always satisfy the other conditions of Lemma~\ref{lem:FHtoAA} we obtain that the action of $G$ on $T$ is acylindrical. Finally, the action will be non-elementary by Theorem~\ref{Thm:categorisedActions} and the fact that the action is non-trivial and $G$ is not virtually cyclic.
\end{proof}

\subsection{The Dunwoody--Sageev Resolution}
In this section, we introduce the main tool of the proof of Theorem~\ref{thm:mainintro}. We begin with the following illustrative example.
\begin{example}\label{Ex:F3}
    Consider the free group $G=F_3=\langle a, b, c\rangle$ on three generators, which is hyperbolic. We will construct an infinitely generated malnormal subgroup of $\langle a, b\rangle\cong F_2$ over which $G$ will split.
    
    Let $(n_i)_{i=1}^\infty$ be the sequence of natural numbers defined inductively by $n_1=1$, $n_{i+1}=456+8n_i^2$.
    Define $x_i=ab^{n_i}ab^{n_i+1}\cdots ab^{n_{i+1}-1}$, and $X_i\subset \langle a, b\rangle$ by $X_i=\{x_1,...,x_i\}$. By choice of the $n_i$s each $X_i$ has the property that $\langle a, b\mid X_i\rangle$ is a $C'(1/6)$ presentation, and so by \cite[Theorem~F]{Logan19} we have that for all $i$, $\langle X_i\rangle\leq \langle a, b\rangle$ is a malnormal free subgroup of $\langle a, b\rangle$ of rank $i$ whose basis is $X_i$. It follows that $H=\langle \bigcup_{i\in \mathbb{N}}X_i\rangle\leq \langle a, b\rangle$ is an infinitely generated malnormal subgroup of $\langle a, b\rangle$. 
    
    We may then decompose $F_3$ as the amalgam $\langle a, b\rangle*_H(H*\langle c\rangle)$. The subgroup $H$ is malnormal in $\langle a, b\rangle$, which is itself malnormal in $G$, so $H$ is malnormal in $G$ and thus has height $2$ in $G$. This decomposition therefore corresponds to a finite height splitting of a hyperbolic group over an infinitely generated edge group. In particular, by Lemma~\ref{lem:QClemma}(1), $H$ is not quasi-convex in $F_3$, and so there exists a finite height splitting of a hyperbolic group which is not a quasi-convex splitting.

    However, it is easy to see that $F_3$ does admit a quasi-convex splitting. Indeed, $F_3\cong \langle a, b\rangle *\langle c\rangle$, a decomposition over the trivial subgroup, and there exists a natural relationship between this splitting and the one considered above. Let $T_1$ be the Bass--Serre tree associated to the decomposition above, and let $T_2$ be the Bass--Serre tree associated to the splitting $F_3\cong\langle a, b\rangle *\langle c\rangle$. Then we may identify every edge in $E(T_1)$ with a coset of $H$, and every edge in $E(T_2)$ with a coset of the trivial subgroup. We use this identification to define the map $\phi:E(T_2)\rightarrow E(T_1)$ that sends the edge labelled by $g$ to the edge labelled by $gH$. This map extends uniquely to a surjective $G$-equivariant graph epimorphism $T_2\rightarrow T_1$.
\end{example}

We may always construct actions with finitely generated edge stabilisers from splittings of finitely presented groups using a construction known as the \emph{Dunwoody--Sageev resolution}.
\begin{theorem}\label{thm:dunwoody}\cite[Lemma~2.2]{DunwoodySageev99}
    If $G$ is a finitely presented group acting by simplicial isometry on a tree $T$ without inversion, then there exists a tree $T'$ such that $G$ acts on $T'$ by simplicial isometry and without inversion with finitely generated edge stabilisers, and such that there exists a $G$-equivariant graph epimorphism    
    \[\phi\colon T'\rightarrow T.\] 
\end{theorem}
In general, if we have two $G$-trees $T'$ and $T$ related by such an epimorphism
we say that $T'$ \emph{dominates} $T$. Dunwoody and Sageev's original proof is constructive, and works by pulling back the midpoints of edges of the tree $T$ onto the Cayley complex associated to some finite presentation of $G$ to give nicely embedded 1-complexes called \emph{tracks}, which we can cut along to give our new splitting. The exact construction will not be of use in this paper, but can be found in \cite{DunwoodySageev99}.
\section{Acylindrical Actions with Well Behaved Stabilisers}
We wish to use the Dunwoody--Sageev resolution to construct quasi-convex (or finite height with finitely generated edge stabilisers) splittings  of hyperbolic groups from acylindrical ones, and as such we are particularly interested in the conditions under which it will preserve acylindricity. 

\subsection{Local Properties of Groups}
The main result of this paper is to prove that if we are given some degree of control over the stabilisers of our action then the Dunwoody--Sageev Resolution will preserve acylindricity. In particular, we will seek to control the local properties of such stabilisers.

\begin{definition}
    Let $G$ be a group, and (P) a property of groups. We say that $G$ locally has property (P) if every finitely generated subgroup of $G$ has property (P).
\end{definition}
\begin{example}
    There are many examples of local properties in the literature, but we will focus on the following two. 
    \begin{enumerate}
        \item Finitely generated free groups, surface groups, and more generally hyperbolic limit groups \cite[Corollary~3.12]{Wilton08} and many small cancellation groups \cite{McCammondWise05, McCammondWise08} are \emph{locally quasi-convex}. It follows that splittings of such groups over finitely generated subgroups will always be quasi-convex splittings, and so if such a group admits any non-trivial action on a tree, then it will admit a quasi-convex splitting arising from the Dunwoody--Sageev resolution.
        \item We say that a group $G$ is \emph{locally finite height} if every finitely generated subgroup of $G$ has finite height in $G$. By Theorem~\ref{thm:QCFH} a locally quasi-convex group will always be locally finite height, but locally finite height groups form a much larger class as they are not required to be hyperbolic, or even finitely generated. For example, the free group on countably many generators is locally finite height owing to its inclusion into $F_2$, which is locally quasi-convex.
    \end{enumerate}
\end{example}

\subsection{Proof of Theorem~\ref{thm:mainintro}}
While requiring local properties on entire groups can be very restrictive, it is much less restrictive to require such properties on the stabilisers of an acylindrical action on a tree. The following is a reformulation of Theorem~\ref{thm:mainintro}, and is our main theorem.
\begin{theorem} \label{thm:main}
    Let $G$ be a finitely presented group with bounded finite subgroups that acts minimally, acylindrically and non-elementarily on a tree $T$. 
    
    \begin{enumerate}[a.]
        \item If all vertex stabilisers of the action of $G$ on $T$ are locally finite height then $G$ admits a non-elementary finite height acylindrical splitting with finitely generated edge stabilisers that dominates the action of $G$ on $T$.
        \item If there exists an edge stabiliser that is finite height in $G$, and this edge stabiliser is locally finite height, then $G$ admits a non-elementary finite height acylindrical splitting with finitely generated edge stabilisers. Furthermore, if every edge stabiliser satisfies this condition then we may assume that the new action dominates the action of $G$ on $T$.
    \end{enumerate}
\end{theorem}
\begin{proof} 
    By definition~\ref{Def:kc} there exist constants $k>0$, $C>1$ such that the action of $G$ on $T$ is $(k, C)$-acylindrical. Furthermore, $G$ has bounded finite subgroups, so we may assume without loss of generality that $C=C_G$, where $C_G$ is the size of the largest finite subgroup of $G$.

    We begin with a proof of point (a). By Theorem~\ref{thm:dunwoody} there exists some tree $T'$ on which $G$ acts with finitely generated edge stabilisers (and therefore finitely generated vertex stabilisers by Lemma~\ref{lem:bieri}) such that $T'$ dominates $T$ along the $G$-equivariant epimorphism $\phi:T'\rightarrow T$. By Lemma~\ref{lem:minimal} the action of $G$ on $T'$ has a unique minimal subtree $T'_m$ with finitely many orbits of edges, and the image of the restriction of $\phi$ to $T'_m$ is a fixed subtree of $T$. This restriction must therefore be surjective by the fact that the action of $G$ on $T$ was minimal, and thus the action of $G$ on $T'_m$ dominates the action of $G$ on $T$. We will therefore assume that $T'=T'_m$ for the remainder of this proof, and we will show that $T'$ corresponds to a non-elementary finite height splitting with finitely generated edge stabilisers.

    Fix a vertex $v\in T$, and consider the set of vertices $U_v=\phi^{-1}(v)\subseteq V(T')$, which we may equip with a metric inherited from $T'$. By equivariance, $G_v=\Ps_G(v)$ acts on $U_v$ by isometry and the action has the following properties:
    \begin{itemize}
        \item[(1)] The $G$-stabiliser $\Ps_G(u)$ of an element $u\in U_v$ with respect to the $G$-action on $T'$ agrees with the $G_v$-stabiliser $\Ps_{G_v}(u)$ with respect to the $G_v$-action on $U_v$. In particular, the $G_v$-stabiliser of each element of $U_v$ is finitely generated; and
        \item[(2)] If $u, u'\in U_v$ are in the same $G$-orbit with respect to the $G$-action on $T'$ then they must be in the same $G_v$-orbit with respect to the $G_v$-action on $U_v$. In particular, the action of $G_v$ on $U_v$ has finitely many orbits.
    \end{itemize}

    Now let $u\in U_v$. The group $\Ps_{G_v}(u)$ is a finitely generated subgroup of $G_v$, so must have finite height in $G_v$ by the assumption that $G_v$ is locally finite height. Define $n_u$ to be the height of $\Ps_{G_v}(u)$ in $G_v$. Pick a set of representatives $S\subset U_v$ of the $G_v$-orbits in $U_v$, and define 
    \[k_v:=1+\sum_{u\in S}n_u,\]
    which is finite by point (2) above. Let $k_{\text{vert}}=\max_{v\in V(T)}(k_v)$, which exists as the action of $G$ on $T$ has finitely many orbits.

    Finally, let $P\subseteq T'$ be a path of length at least $k'=(2k_{\text{vert}})^k\cdot k_{\text{vert}}$. We then have two cases, the first of which occurs when $\phi(P)$ has diameter at least $K$. Then by the $(k, C)$-acylindricity of the action of $G$ on $T$ we have that $|\Ps_G(P)|\leq |\Ps_G(\phi(P)|\leq C$.

    The second case is where the interest lies, and occurs when the diameter of $\phi(P)$ is strictly less than $k$. We make the following claim.
    
    \begin{claim} \label{claim}
        There exists some vertex $v\in\phi(P)$ such that $|\phi^{-1}(v)\cap P|\geq k_{\text{vert}}$.
    \end{claim} 
    
    Indeed, for all $v\in \phi(P)$ we have that $|\phi^{-1}(v)\cap P|>\operatorname{deg}_{\phi(P)}(v)/2$, where $\operatorname{deg}_{\phi(P)}$ is the degree of the vertex $v$ when considered as a vertex of the finite graph $\phi(P)$. Thus, either $\phi(P)$ contains a vertex of degree at least $2k_{\text{vert}}$ and we are done, or no such vertex exists and $\phi(P)$ has degree bounded above by $2k_{\text{vert}}$. The graph $\phi(P)$ is a tree of diameter at most $k$, so it must then contain at most $(2k_{\text{vert}})^k$ vertices. But then by the pigeonhole principle there exists at least one vertex whose pre-image contains at least $k_{\text{vert}}$ vertices of $P$, so the claim is proven.

    Fix such a $v$. We have that $|\phi^{-1}(v)\cap P|>k_{\text{vert}}>k_v$, so by our choice of $k_v$ and by the pigeonhole principle there exists some orbit $G_v\cdot u$ of vertices in $\phi^{-1}(v)$ such that $|\phi^{-1}(v)\cap P\cap (G_v\cdot u)|\geq n_u$, the height of $\Ps_{G_v}(u)$ in $G_v$. It follows that, since each vertex in $\phi^{-1}(v)\cap P\cap (G_v\cdot u)$ is stabilised (with respect to the $G$-action on $T'$) by a $G_v$-conjugate of $\Ps_{G_v}(u)$, the pointwise $G$-stabiliser of this set must be finite, and therefore so must the $G$-stabiliser $\Ps_G(P)$ of $P$. 

    The action of $G$ on $T'$ is therefore $(k', C_G)$-acylindrical, and thus it only remains to show that the action of $G$ on $T'$ is a finite height splitting. For all $v\in V(T')$ the stabiliser $\Ps_G(v)$ is a subgroup of the stabiliser $\Ps_G(\phi(v))$, and so in particular is locally finite height. Thus, as each edge stabiliser is finitely generated it must have finite height in both of its neighbours, and the action of $G$ on $T'$ will be a finite height splitting by Lemma~\ref{LFHtoFH} as required. It only remains to check that the action of $G$ on $T'$ is non-elementary. The group $G$ is acylindrically arboreal by assumption, so $G$ is not virtually cyclic and the action of $G$ on $T'$ cannot be lineal by Theorem~\ref{Thm:categorisedActions}. The action of $G$ on $T'$ can also have no fixed points, as it dominates the action of $G$ on $T$ with has no fixed points, and so the action of $G$ on $T'$ will be non-elementary by Theorem~\ref{Thm:categorisedActions} as required.

    \bigskip We now move onto a proof of part b of this theorem. We begin by modifying $T$ by collapsing the quotient graph of groups to consider only the splitting over the edges $e\in E(\Gamma)$ with the properties required in the hypothesis of part b, and call this new graph of groups $(\Gamma_c, \mathfrak{G}_c)$. This will still correspond to a minimal non-elementary acylindrical splitting by \cite[Lemma~2.20]{Cohen25}, with associated Bass-Serre tree $T_c$. Construct the tree $T'$ as before using Theorem~\ref{thm:dunwoody} applied to $T_c$, which we will again assume to be minimal with a finite number of orbits of edges. Therefore $T'$ will dominate our original splitting if no edge of $\Gamma$ was collapsed, as required. Again as before, fix some $e\in E(T_c)$, and consider the set of edges $S_e=\phi^{-1}(e)\subseteq E(T')$, which we may equip with a metric given by considering the distance between their midpoints in $T'$. By equivariance, $G_e=\Ps_G(e)$ acts on $S_e$ by isometry and the action has the properties analogous to property (1) above.

    Now let $f\in E(T')$. By property (1), the stabiliser $\Ps_G(f)$ agrees with the stabiliser $\Ps_{G_{\phi(f)}}(f)$, which is therefore a finitely generated subgroup of $G_{\phi(f)}$ and thus has finite height in $G_{\phi(f)}$. In turn, $G_{\phi(f)}$ has finite height in $G$ by hypothesis, and so by transitivity of the finite height property the stabiliser $\Ps_G(f)$ must have finite height in $G$. The action of $G$ on $T'$ is therefore a finite height splitting, and will therefore be acylindrical by Lemma~\ref{lem:FHtoAA}. Finally, the action of $G$ on $T'$ will be non-elementary as required as in the proof of part a.

\end{proof}
The power of this Theorem~\ref{thm:main} lies in the fact that if $G$ is a group (even one that is not finitely presented or even finitely generated) acting on a tree $T$ and this action satisfies the hypotheses of Theorem~\ref{thm:main} then the induced action of any finitely presented subgroup of $G$ on $T$ will often satisfy the same hypotheses. We have the following corollary.
\begin{corollary}\label{cor:main}
    Let $G$ be a group acting acylindrically and non-elementarily on a tree $T$, and $H$ a finitely presented subgroup of $G$ with bounded finite subgroups acting non-trivially on $T$ such that $H$ is not virtually cyclic. Let $T_H$ be the minimal subtree of the induced action of $H$ on $T$ guaranteed by Lemma~\ref{lem:minimal}.
    \begin{enumerate}[a.]
        \item If all vertex stabilisers of the action of $G$ on $T$ are locally finite height, then $H$ admits a non-elementary finite height acylindrical splitting with finitely generated edge stabilisers that dominates the induced action of $H$ on $T_H$.
        \item If there exists an edge $e\in E(T_H)$ whose $G$-stabiliser (with respect to the action of $G$ on $T$) has finite height in $G$, and this edge stabiliser is locally finite height, then $H$ admits a non-elementary finite height acylindrical splitting with finitely generated edge stabilisers. Furthermore, if every edge $e\in E(T_H)$ satisfies this condition then we may assume that the new action dominates the induced action of $H$ on $T_H$.
    \end{enumerate}
\end{corollary}
\begin{proof}
    If the hypotheses of part~(a) of this corollary hold then the action of $H$ on $T_H$ will satisfy the hypotheses of Theorem~\ref{thm:main} part~(a), with the non-elementary condition arising from Theorem~\ref{Thm:categorisedActions} and the facts that $H$ is not virtually cyclic and the action of $H$ on $T$ is non-trivial.

    Similarly, if the hypotheses of part~(b) of this corollary hold then the action of $H$ on $T_H$ will satisfy the hypotheses of Theorem~\ref{thm:main}(b).
\end{proof}

\subsection{Application to Hyperbolic One-Relator Groups} \label{sec:Applications}
We now apply the methods used in the proof of Theorem~\ref{thm:main} to Question~\ref{Q1} via one-relator groups.
\begin{definition}
 Let $G$ be a one-relator group, and $\langle X\mid w\rangle$ a one-relator presentation for $G$ with $w$ cyclically reduced as a word over $X$. We say that a subgroup $H\leq G$ is a \emph{Magnus subgroup with respect to this presentation} if there exists some subset $Y\subset X$ such that $w\notin\langle Y\rangle< F(X)$, and $Y$ generates $H$ in $G$. We say that a subgroup is a \emph{Magnus subgroup} if it is a Magnus subgroup relative to some one-relator presentation of $G$.
\end{definition}
    
One of the oldest and most powerful methods with which to understand one-relator groups is with a \emph{Magnus hierarchy}.

\begin{definition}\label{def:Magnus}
    Let $G$ be a one-relator group. We say that $G_0\hookrightarrow G_1\hookrightarrow\cdots\hookrightarrow G_n\cong G$ is a \emph{Magnus hierarchy} for $G$ if:
    \begin{enumerate}[(M1)]
    \item Each $G_i$ is a one-relator group, and $G_0$ is isomorphic to the free product of finitely many cyclic groups, and;
    \item for each $i$, $G_{i+1}$ is an HNN-extension of $G_i$ along the identification of two subgroups that are both Magnus subgroups with respect to the same presentation of $G_i$.
    \end{enumerate}
\end{definition}
In the language of Definition~\ref{def:hier}, a Magnus hierarchy for a one-relator group $G$ can be identified with a hierarchy decomposition of length $n+1$:
\begin{itemize}
    \item For all $i\neq 0$, set $J_i=\{1\}$ and let $G_{i, 1}\cong G_{i-1}$. Then for $i\geq 2$, $T_{i, 1}$ may be chosen to be the $HNN$-extension decomposition of $G_{i-1}$ with edge stabilisers isomorphic to $G_{i-1}$ coming from the Magnus hierarchy. 
    \item The group $G_{1, 1}\cong G_0$ is the free product of finitely many cyclic groups, so acts on a tree with finite vertex stabilisers by \cite{KarasPietrowskiSolitar73}. We may therefore choose $T_{1, 1}$ to be this tree to satisfy the condition in Definition~\ref{def:hier} that all groups $G_{0, j}$ for $j\in J_0$ must be finite.
\end{itemize} 

\begin{theorem}(\cite{Masters06}, see also \cite{Magnus30, Moldavanskii67})
    Let $G$ be a one-relator group. Then $G$ admits a Magnus hierarchy.
\end{theorem}

Returning to acylindricity, we have the following powerful result due to Linton, where the definition of $\Z$-stability can be found in \cite[Section~1.3]{Linton25}.

\begin{theorem}\label{thm:Linton}(\cite[Theorem~7.1]{Linton25})
    Let $G$ be a one-relator group and $G_0\hookrightarrow G_1\hookrightarrow\cdots\hookrightarrow G_n\cong G$ a Magnus hierarchy for $G$. Then the following are equivalent.
    \begin{enumerate}
        \item The hierarchy $G_0\hookrightarrow G_1\hookrightarrow\cdots\hookrightarrow G_n\cong G$ is an acylindrical hierarchy, i.e every HNN-extension in the hierarchy is an acylindrical splitting.
        \item The group $G$ is hyperbolic and the hierarchy $G_0\hookrightarrow G_1\hookrightarrow\cdots\hookrightarrow G_n\cong G$ is a quasi-convex hierarchy, i.e every HNN-extension in the hierarchy is a quasi-convex splitting.
        \item The group $G$ has no Baumslag--Solitar subgroups and the hierarchy $G_0\hookrightarrow G_1\hookrightarrow\cdots\hookrightarrow G_n\cong G$ is $\Z$-stable.
    \end{enumerate}
\end{theorem}

Now let $H$ be a finitely generated subgroup of a one-relator group $G$ with an acylindrical Magnus hierarchy. The group $H$ will admit an acylindrical action on a tree arising from the hierarchy for $G$. However, there is no guarantee that such a splitting will have finitely generated edge stabilisers. Consider the following.

\begin{example} \label{ex:BNS}
    Let $G$ be the two-generator one-relator group with the presentation \[G=\langle a, b\mid a^2ba^{-1}b^2a^{-2}ba^3b^{-2}a^{-1}ba^{-2}b^{-2}aba^{-1}b^{-1}ab^{-1}\rangle.\] It may be shown by direct calculation that this presentation is $C^\prime (1/6)$, and thus $G$ is hyperbolic. Using the homomorphism $G\rightarrow\Z$ that sends $a$ to $1$ and $b$ to $0$ we may express $G$ as the HNN-extension of the one-relator group \[G_1=\langle c, d, e, f\mid fe^2cf^{-2}ec^{-2}dc^{-1}d^{-1}\rangle\] along the subgroup isomorphism $\phi_1:\langle c, d, e\rangle\mapsto\langle d, e, f\rangle $, and in turn we may use the homomorphism $G_1\rightarrow\Z$ that sends $b$ to $1$ and all other generators to $0$ to express $G_1$ as the HNN-extension of the free group \[G_2= \langle g, h, i, k, l, m\mid ih^2gi^{-2}hg^{-2}k\rangle \cong F(g, h, i, l, m)\] along the subgroup isomorphism $\phi_2:\langle g, h, i\rangle\mapsto\langle k, l, m\rangle $. This is therefore a Magnus hierarchy for $G$, which we claim is acylindrical. Indeed, one may verify by hand or using the pseudo-algorithm of Linton \cite[Proposition~6.1.7]{LintonThesis} that both extensions are $\Z$-stable, and the fact that $G$ is hyperbolic implies that it is Baumslag-Solitar free. Thus, by Theorem~\ref{thm:Linton} this hierarchy is acylindrical, and is in turn a quasi-convex hierarchy for $G$.

    Let $A$ be some edge group of the first splitting in this hierarchy, so by construction the normal closure of $A$ is the kernel of the homomorphism $G\rightarrow\Z$ that sends $a$ to $1$ and $b$ to $0$. One may verify that this is not finitely generated using Brown's algorithm \cite[Theorem~4.3]{BrownBNS}. However, the group $G$ does fibre --- by the same algorithm the homomorphism $\chi:G\rightarrow \Z$ that sends $a$ to $1$ and $b$ to $-1$ has finitely generated kernel, which we will call $F$. 
    
    Consider now the intersection of $F$ with any $G$-conjugate $A^g$ of $A$. We must have that $A^g\not\subset F$ because the normal closure of $A$ is the infinitely generated kernel of a map to $\Z$, and so the quotient of $A^g$ by $F\cap A^g$ must be isomorphic to $\Z$. It follows that, since $A^g$ is free of rank $3$, $F\cap A^g$ cannot be finitely generated. The splitting of $F$ induced by the above hierarchy therefore does not have finitely generated edge groups, but will be acylindrical (and indeed finite height by Theorem~\ref{thm:Linton}).
\end{example}
While the group $F$ in Example~\ref{ex:BNS} is a non-elementary free group, so naturally admits a free splitting, the above argument illustrates that acylindrical graphs (or hierarchies) of groups, even in the relatively tame case of one-relator groups, may not induce splittings with finitely generated edge groups on their subgroups. However, we may still deal with such subgroups using methods from the proof of Theorem~\ref{thm:main}. We have the following corollary, which is a reformulation of Corollary~\ref{cor:introStrong}.

\begin{corollary}
    Let $G$ be a hyperbolic group, and let $H$ be a hyperbolic subgroup of $G$.
    
    \begin{enumerate}[(a)]
        \item If $G$ admits an action on a tree $T$ that is a quasi-convex splitting over a locally quasi-convex group, and $H$ the induced action of $H$ on $T$ is non-trivial, then $H$ admits a non-trivial quasi-convex splitting that dominates the minimal subtree of the induced action of $H$ on $T$.
        \item If $G$ admits a quasi-convex hierarchy with locally quasi-convex edge stabilisers then $H$ admits a quasi-convex hierarchy.
    \end{enumerate}
\end{corollary}
\begin{proof}
    We begin with part a. Let $T$ be a tree on which $G$ acts such that the action of $G$ on $T$ is a quasi-convex splitting, and let $T_H$ be the minimal subtree of the induced action of $H$ on $T$ guaranteed by Lemma~\ref{lem:minimal}.  By Theorem~\ref{thm:dunwoody} there exists some tree $T'$ on which $H$ acts with finitely generated edge stabilisers such that $T'$ dominates $T_H$ along the $H$-equivariant epimorphism $\phi:T'\rightarrow T_H$. By Lemma~\ref{lem:minimal} the action of $H$ on $T'$ has a unique minimal subtree $T'_H$ with finitely many orbits of edges, and the image of the restriction of $\phi$ to $T'_H$ is a fixed subtree of $T_H$. This restriction must therefore be surjective by the fact that the action of $H$ on $T_H$ was minimal, and thus the action of $H$ on $T'_H$ dominates the action of $G$ on $T$. As in the proof of Theorem~\ref{thm:main}, fix some $e\in E(T_H)$ and consider the set of edges $S_e=\phi^{-1}(e)\subseteq E(T_H')$, which we may equip with a metric given by considering the distance between their midpoints in $T_H'$. By equivariance, $H_e=\Ps_H(e)$ acts on $S_e$ by isometry and this action has the properties analogous to property (1) in the proof of Theorem~\ref{thm:main}(a). In particular, the $H$-stabiliser $\Ps_H(f)$ of an element $f\in S_e$ with respect to the $H$-action on $T'_H$ agrees with the $H_e$-stabiliser $\Ps_{H_e}(f)$ with respect to the $H_e$-action on $S_e$. 
    
    The $H$-stabiliser of each element of $f\in E(T_H')$ is therefore a finitely generated subgroup of $H\cap G_{\phi(e)}$, where $G_{\phi(e)}=\Ps_G(e)$ is the $G$-stabiliser of $e$ with respect to the $G$-action on $T$. Furthermore, by assumption on the action of $G$ on $T$, $G_e$ is a locally quasi-convex quasi-convex subgroup of $G$, so every finitely generated subgroup of $G_e$ is a quasi-convex subgroup of $G$ by transitivity of quasi-convexity. It therefore follows that each stabiliser of the action of $H$ on $T_H'$ is a quasi-convex subgroup of $G$, so is a quasi-convex subgroup of each hyperbolic subgroup of $G$. In particular, each stabiliser of the action of $H$ on $T_H'$ is a quasi-convex subgroup of $H$ as required.

    \bigskip We now proceed with part b. Let $n$ a non-negative integer, and for all $0\leq i< n$ let $J_i$ be a finite indexing set with $J_n=\{1\}$. Let $\{G_{i, j}\}_{0\leq i\leq n, j\in J_i}$ be the set of groups in a quasi-convex hierarchy decomposition for $G$ of length $n$, as in Definition~\ref{def:hier}. We may assume without loss of generality that $H$ is not isomorphic to a subgroup of any element of $\{G_{i, j}\}_{0\leq i\leq n, j\in J_i}$ other than $G_{n, 0}$, and proceed by induction on $n$. First assume $n=0$. Then $G$ is finite, and so $H$ is finite and admits a quasi-convex hierarchy of length $0$ as required.

    Now assume that the above result holds for all hyperbolic subgroups of hyperbolic groups $G$ admitting a quasi-convex hierarchy  of length at most $k-1$ with locally quasi-convex edge stabilisers, and assume that $n=k$. By property (H3) in Definition~\ref{def:hier}, $G\cong G_{n, 1}$ admits an action on a tree $T$ with quasi-convex edge stabilisers and with vertex stabilisers of the form $G_{n-1, k}$ for some $k\in J_{n-1}$, and by assumption $H$ is not isomorphic to a subgroup of any vertex stabiliser of this action, so $H$ acts non-trivially on $T$. We may therefore apply part a of this corollary to see that $H$ admits a quasi-convex splitting that dominates the minimal subtree of the induced action of $H$ on $T$. The vertex stabilisers of this splitting will be subgroups of the vertex stabilisers of the action of $G$ on $T$ as above, and will be quasi-convex in $H$ by Lemma~\ref{lem:BowditchVertex} and thus hyperbolic by Lemma~\ref{lem:QClemma}(a). All vertex stabilisers of the action of $G$ on $T$ are of the form $G_{n-1, k}$ for some $k\in J_{n-1}$, and thus admit quasi-convex hierarchies of length at most $n-1$, and so it follows by the inductive hypothesis that each stabiliser of the action of $H$ on $T_H'$ admits a quasi-convex hierarchy. It follows that $H$ must admit a quasi-convex hierarchy as required.
\end{proof}
The following corollary  is a reformulation of Corollary~\ref{cor:intro}, and therefore answers Question~\ref{Q1} in the affirmative for all finitely generated subgroups of one-relator groups with acylindrical hierarchies.
\begin{corollary}
    Let $H$ be a finitely generated subgroup of a one-relator group $G$ such that $G$ admits an acylindrical Magnus hierarchy $G_0\hookrightarrow G_1\hookrightarrow\cdots\hookrightarrow G_n\cong G$. Then $H$ admits a quasi-convex hierarchy. In particular the group $H$ is acylindrically arboreal if and only if it admits a non-elementary quasi-convex splitting.
\end{corollary}
\begin{proof}
    First, we observe that by Theorem~\ref{thm:Linton} $G$ is hyperbolic and $G_0\hookrightarrow G_1\hookrightarrow\cdots\hookrightarrow G_n\cong G$ is a quasi-convex hierarchy whose edge stabilisers are free, and therefore locally quasi-convex by Halls' theorem \cite{Hall_1949}. Furthermore, by \cite[Corollary~9]{Howie84} $G$ has virtual cohomological dimension at most two, so there exists some finite index subgroup $G'\leq G$ such that $\operatorname{cd}_\Z(G')\leq 2$. By coherence of one-relator groups\cite[Theorem~1.1]{JaikinLinton25} $H$ will be finitely presented, and $H'=H\cap G'$ is a finite index subgroup of $H$ so must also be finitely presented. It follows by Gersten's theorem \cite[Theorem~5.4]{Gersten96} that $H'$ is hyperbolic, and so $H$ must also be hyperbolic. We are therefore able to apply Corollary~\ref{cor:introStrong} to see that $H$ admits a quasi-convex hierarchy as required. 

    Now assume that $H$ is acylindrically arboreal. It is therefore not virtually cyclic by Theorem~\ref{Thm:categorisedActions}, and so the hierarchy given by Corollary~\ref{cor:introStrong} in the first part of this proof will not be trivial. Taking the first stage of this hierarchy therefore gives the required non-trivial finite-height splitting with finitely generated edge stabilisers.

    Finally, assume $H$ is not virtually cyclic and admits a non-trivial finite-height splitting with finitely generated edge stabilisers. This splitting may be assumed to be minimal with finitely many orbits of edges by Lemma~\ref{lem:minimal}, and $H$ had bounded torsion as a subgroup of the hyperbolic group $G$, so it will be acylindrical by Lemma~\ref{lem:FHtoAA}. This splitting will also not be lineal as $H$ is not virtually cyclic, and so by Theorem~\ref{Thm:categorisedActions} will be non-elementary acylindrical. It follows that $H$ is acylindrically arboreal as required.
\end{proof}

\subsection{Relaxing The Conditions on Stabilisers} \label{sec:Broken}
We finish this paper by demonstrating the danger of attempting to extend these results to more general action. The following is a reformulation of Proposition~\ref{prop:introBadExample}
\begin{proposition}\label{prop:badExample}
     There exists a finitely presented group $G$ such that:
    \begin{enumerate}
        \item $G$ acts acylindrically and non-elementarily on a tree $T$ with infinitely generated edge stabilisers; and
        \item if $T'$ is a tree on which $G$ acts acylindrically with $T'$ dominating $T$ then the action of $G$ on $T'$ has at least one infinitely generated edge stabiliser.
    \end{enumerate}
\end{proposition}
\begin{proof}
Let $G_0 = F(a, b)$ be the free group on generators $a$ and $b$, and let $H_0$ be an infinitely generated malnormal subgroup of $G_0$, for example we may choose $H_0$ to be the subgroup $H$ from Example~\ref{Ex:F3}. Let $G_1$ be the trivial HNN-extension of $G_0$ along $H_0$ with stable letter $t_0$, so $G_1=\langle G_0, t_0\mid [t_0, H_0]\rangle$, and let $H_1=\langle H_0, t_0\rangle\leq G_1$. Thus $H_1\cong \mathbb{F}_\mathbb{N}\times \Z$.

We claim that $H_1\leq G_1$ is malnormal. Indeed, consider the action of $G_1$ on the Bass--Serre tree $T_1$ of the splitting $G_0*_{H_0}H_1$, and let $g\in G_1$. Then either $g\in H_1$ or $gH_1$ is a separate coset of $H_1$ in $G_1$, so represents a different vertex of $T_1$. Assume that we are in the latter case. Then $gH_1g^{-1}\cap H_1$ stabilises the path between $H_1$ and $gH_1$ in $T_1$, which passes through a vertex labelled with some coset $f G_0$ of $G_0$. It follows that $gH_1g^{-1}\cap H_1$ is contained in the intersection of stabilisers of two distinct edges leaving $f G_0$, which must be trivial by the malnormality of $H_0$ in $G_0$. Thus $gH_1g^{-1}\cap H_1$ is trivial, and $H_1$ is malnormal as required.

The group $G_1$ is recursively presented, so there exists a finitely presented group $G_2$ such that $G_1\leq G_2$ is a malnormal subgroup by a result of Wagner \cite[Theorem~1.1]{Wagner24}. Define $G$ to be the trivial HNN-extension of $G_2$ along $\langle t_0\rangle$, so $G=\langle G_2, t\mid [t, t_0]\rangle$.

By construction, $G$ splits as $G_2*_{\langle t_0\rangle}\Z^2$, so in particular is finitely presented. Furthermore, we may fold this splitting to give the amalgam decomposition
\[G=G_2*_{H_1} (H_1*_{\langle t_0\rangle}\Z^2),\]
and define $T$ to be the Bass--Serre tree of this splitting.
The subgroup $H_1$ is malnormal in $G_1$, which itself is malnormal in $G_2$, so it follows that $H_1$ is malnormal in $G_2$ and this splitting is (3, 1)-acylindrical. Since $G$ is clearly not virtually cyclic and this splitting is non-trivial by construction it follows that the action of $G$ on $T$ is non-elementary.

Assume now that $T'$ is a tree on which $G$ acts acylindrically that dominates $T$ via the $G$-equivariant epimorphism
\[\phi\colon T'\rightarrow T.\]
Let $v$ be the unique vertex of $T$ whose stabiliser is $(H_1*_{\langle t_0\rangle}\Z^2)$, and let $\overline{\phi^{-1}(v)}$ be the convex hull of the pre-image set of $v$ in $T'$, which will be a subtree of $T'$. The group $(H_1*_{\langle t_0\rangle}\Z^2)$ acts on $\overline{\phi^{-1}(v)}$ acylindrically by assumption, and the fact that $\langle t_0\rangle\cong \Z$ is normal in both $H_1$ and $\Z^2$ (and hence in $(H_1*_{\langle t_0\rangle}\Z^2)$) implies that $(H_1*_{\langle t_0\rangle}\Z^2)$ is not acylindrically arboreal by Lemma~\ref{lem:AAnormal}, so the action of $(H_1*_{\langle t_0\rangle}\Z^2)$ on $\overline{\phi^{-1}(v)}$ must be elementary. Indeed, given that $(H_1*_{\langle t_0\rangle}\Z^2)$ is not virtually cyclic the action must be elliptic, and there exists some vertex $\Tilde{v}\in \overline{\phi^{-1}(v)}$ fixed by the action of $(H_1*_{\langle t_0\rangle}\Z^2)$.

We will show that $\Ps_G(\Tilde{v})=(H_1*_{\langle t_0\rangle}\Z^2)$, and thus that at least one vertex stabiliser of the action of $G$ on $T'$ must be infinitely generated. If $g\in (H_1*_{\langle t_0\rangle}\Z^2)$ then
\[g\cdot\phi(\Tilde{v})=\phi(g\cdot\Tilde{v})=\phi(\Tilde{v}),\]
so $(H_1*_{\langle t_0\rangle}\Z^2)\leq \Ps_G(\phi(\Tilde{v}))$. By definition of $T$ this will hold only when $\phi(\Tilde{v})=v$. It follows that $(H_1*_{\langle t_0\rangle}\Z^2)\leq\Ps_G(\Tilde{v})\leq (H_1*_{\langle t_0\rangle}\Z^2)$, so $\Ps_G(\Tilde{v})= (H_1*_{\langle t_0\rangle}\Z^2)$  as required, and that at least one vertex stabiliser of the action of $G$ on $T'$ is infinitely generated. The result then follows by Lemma~\ref{lem:bieri}.
\end{proof}
\begin{corollary}
    The Dunwoody--Sageev resolution applied to the acylindrical splitting of the group $G$ above will not preserve acylindricity.
\end{corollary}
The example we construct here is not hyperbolic, and we are unaware of such an example which is, although we believe one is likely to exist. It is also the case that the group $G$ constructed in the proof of Proposition~\ref{prop:badExample} does have a non-elementary action on a tree with finitely generated edge stabilisers, given by the amalgam decomposition
\[ G\cong G_2*_{G_1}\left(G_1*_{\langle t_0\rangle}\Z^2\right),\]
which is acylindrical due to the malnormality of $G_1$ in $G_2$. We do however believe that there will be an example of a finitely presented group that acts acylindrically on a tree but not with finitely generated edge stabilisers, and finish this paper by restating Question~\ref{q:mine}.

\begin{question}
    Does every acylindrically arboreal finitely presented group admit a non-elementary acylindrical action on a simplicial tree with finitely generated edge stabilisers?
\end{question}

\printbibliography
 \end{document}